\documentclass[11pt]{article}
\usepackage{amsmath,amssymb,amsbsy,amsfonts,amsthm,latexsym,
            amsopn,amstext,amsxtra,euscript,amscd,amsthm}

\newtheorem{lem}{Lemma}
\newtheorem{lemma}[lem]{Lemma}

\newtheorem{thm}{Theorem}
\newtheorem{theorem}[thm]{Theorem}

\newtheorem{defi}{Definition}

\def\\{\cr}
\def\({\left(}
\def\){\right)}
\def\[{\left[}
\def\]{\right]}
\def\<{\langle}
\def\>{\rangle}

\def\cM{{\mathcal M}}

\def\cP{{\mathcal P}}

\begin{document}

\title{Counting terms $U_n$ of third order linear recurrences with $U_n=u^2+nv^2$}

\author{{\sc Emil-Alexandru Ciolan} \\
Rheinische Friedrich-Wilhelms-Universit\"at Bonn\\
Regina Pacis Weg 3\\
D-53113 Bonn, Germany\\
{\tt calexandru92@yahoo.com}
\and
{\sc Florian~Luca}\\ 
{School of Mathematics}\\ 
{University of the Witwatersrand}\\ 
{Private Bag 3, Wits 2050, South Africa}\\
{florian.luca@wits.ac.za}\\
\and
{\sc Pieter Moree}  \\
Max Planck Institut f\"ur Mathematik\\
Vivatsgasse 7\\
D-53111 Bonn, Germany\\
{\tt moree@mpim-bonn.mpg.de}
}

\date{\today}

\pagenumbering{arabic}

\maketitle

\begin{abstract}
Given a recurrent sequence ${\bf U}:=\{U_n\}_{n\ge 0}$ we consider the problem of counting
${\mathcal M}_U(x)$, the number of integers $n\le x$ such that $U_n=u^2+nv^2$ for some integers $u,v$.
We will show that  ${\mathcal M}_U(x)\ll x(\log x)^{-0.05}$ for a large class of ternary sequences.
Our method uses many ingredients from the proof of Alba Gonz\'alez and the second author \cite{AL} that
${\mathcal M}_F(x)\ll  x(\log x)^{-0.06}$, with $\bf F$ the Fibonacci sequence.
\end{abstract}

\section{Introduction}
If ${\mathcal M}$ is a set of real numbers and $x$ a positive real number $x$, we 
put ${\mathcal M}(x)={\mathcal M}\cap [1,x)$. 
Given a recurrent sequence ${\bf U}:=\{U_n\}_{n\ge 0}$ 
we put 
\begin{equation}
\label{muuuh}
{\mathcal M}_U=\{n: U_n=u^2+nv^2 ~{\text{\rm for~some~integers}}~u,~v\}.
\end{equation}
\indent Let ${\bf F}:=\{F_n\}_{n\ge 0}$ be the Fibonacci sequence given by $F_0=0,~F_1=1$ and 
$$
F_{n+2}=F_{n+1}+F_n\quad {\text{\rm for~ all}}\quad n\ge 0.
$$ 
Some results concerning Fibonacci numbers which can be represented by certain positive definite quadratic forms in two variables 
appear in \cite{Sav}. In \cite{ABL} it was shown that if $p\equiv 1\pmod 4$ is a prime, then $F_p=u^2+pv^2$ for some integers $u$ and $v$. 
It follows from the prime number theorem in arithmetic progressions that 
$\# {\mathcal M}_F(x)\gg x/\log x$. In \cite{AL} it was shown that 
$$
\#{\mathcal M}_F(x)\ll \frac{x}{(\log x)^{0.06}}.
$$
In this paper we use the method from \cite{AL} to study the analogous problem for certain third order linearly recurrent sequences  ${\bf U}:=\{U_n\}_{n\ge 0}$ of integers. Assume that $U_0,~U_1,~U_2\in {\mathbb Z}$ and that 
$$
U_{n+3}=a_1 U_{n+2}+a_2 U_{n+1}+a_3 U_n\quad {\text{\rm for~all}}\quad n\ge 0,
$$
where $\Psi_U(X)=X^3-a_1X^2-a_2X-a_3\in {\mathbb Z}[X]$. 
Let
$$\Psi_U(X)=(X-\alpha)(X-\beta)(X-\gamma)$$
be the factorization of $\Psi_U$ over the complex numbers.
We assume that $a_3\ne 0$. Let ${\mathbb K}$ be the splitting field of $\Psi_U$ over ${\mathbb Q}$ and $G$ be its Galois group. We assume that the following conditions are fulfilled:
\begin{itemize}
\item[(i)] $G$ contains a transposition (as a subgroup of $S_3$).
\item[(ii)] Either $a_3=\pm 1$ and $\Psi_U(X)$ is irreducible  over ${\mathbb Q}$, or 
$$\Psi_U(X)=(X-a)(X^2+bX+c),\quad {\text{\rm where}}\quad a\in {\mathbb Z}\backslash\{\pm 1\}\quad {\text{\rm  and}}\quad c=\pm 1.
$$
\item[(iii)] The ratio of any two roots of $\Psi_U(X)$ is not a root of unity. 
\end{itemize}
In case $\Psi_U(X)$ is irreducible over ${\mathbb Q}$, its constant coefficient is $a_3=\pm 1$ and $G$ is a transitive subgroup of $S_3$. Condition (i) ensures that this group cannot be isomorphic to ${\mathbb Z}/3{\mathbb Z}$, 
therefore it must be $S_3$. This is equivalent to the condition that the discriminant of $\Psi_U(X)$, which is 
$$
a_1^2 a_2^2+4a_2^3-4a_1^3 a_3-18 a_1a_2a_3-27 a_3^2,
$$
is not the square of an integer. In case $\Psi_U(X)$ is not irreducible over ${\mathbb Q}$, then the combination of conditions (i) and (ii) above ensures that $\Psi_U(X)$ has exactly 
one integer root $a$ which is not $\pm 1$, and the other two roots are quadratic units. In that case, ${\mathbb K}$ is a quadratic field and the nonidentity element of $G$ fixes $a$
and switches the other two roots, so this can be regarded as a transposition in $S_3$.

We give two examples of sequences satisfying our conditions, formulate our main result, and then give three examples for which the conclusion of our theorem do not hold and compare them with (i), (ii) and (iii) above. 

Recall that the Tribonacci sequence ${\bf T}:=\{T_n\}_{n\ge 0}$ is defined as $T_0=T_1=0$, $T_2=1$ and 
$$
T_{n+3}=T_{n+2}+T_{n+1}+T_n\quad {\text{\rm for~all}}\quad n\ge 0.
$$
In this case, $\Psi_T(X)=X^3-X^2-X-1$ is irreducible over ${\mathbb Q}$ and its Galois group is $S_3$. So, our result applies to the Tribonacci sequence. 
Another sequence to which it applies is the sequence ${\bf U}$ of numbers of the form $U_n=2^n+F_n$, where $\{F_n\}_{n\ge 0}$ is the Fibonacci sequence. This is ternary recurrent with characteristic polynomial
$$
\Psi_U(X)=(X-2)(X^2-X-1),
$$
which satisfies the conditions (i), (ii) and (iii).
\begin{theorem}
\label{thm:1}
Assume that ${\bf U}:=\{U_n\}_{n\ge 0}$ is a ternary recurrent sequence satisfying (i), (ii) and (iii). Then the following estimate holds
$$
\#\{n\le x: U_n=u^2+nv^2 ~{\text{\rm for~some~integers}}~u,~v\}\ll \frac{x}{(\log x)^{0.05}}.
$$
\end{theorem}
Note that the conditions of the theorem depend only on the characteristic polynomial of $U$. Thus if
$a\ge 0$ is an integer and we ask for the number of solutions of $U_{n+a}=u^2+nv^2$ the estimate above also
holds. More informally, we could say that our result is robust under relabelling of the sequence.

While conditions (i), (ii) and (iii) can perhaps be weakened, some conditions have to be imposed on ${\bf U}:=\{U_n\}_{n\ge 0}$ in order to conclude that the set of positive integers $n$ such that $U_n=\square+n\square$ is of density zero. 
Indeed, consider the three examples
$$
U_n=2^n+n,\qquad U_n=4^n+2^{n+1}+1\qquad {\text{\rm and}}\qquad U_n=5F_n^2-4.
$$
In the first case, $U_n=\square +n\square$ for all $n$ even. In the second case, $U_n=\square$ holds for all $n\ge 0$. In the third case, $5F_n^2-4=L_n^2=\square$ holds for all odd $n$, where $\{L_n\}_{n\ge 0}$ is the companion sequence of the Fibonacci sequence given by $L_0=2,~L_1=1$ and $L_{n+2}=L_{n+1}+L_n$ for all $n\ge 0$. Thus, in all the above cases $U_n=\square+n\square$ holds for a positive proportion of $n$, where in the last two cases the second $\square$ (which multiplies $n$) is zero. Note that for the first two sequences 
$$
\Psi_U(X)=(X-2)(X-1)^2\quad \text{and}\quad \Psi_U(X)= (X-4)(X-2)(X-1),
$$
respectively, so that $\Psi_U(X)$ factors completely over ${\mathbb Q}$ and in the first case it even has a double root, whereas for the third sequence, we have 
$$
\Psi_U(X)=(X+1)(X^2-3X+1),
$$
for which ${\mathbb K}={\mathbb Q}({\sqrt{5}})$, so condition (i) is satisfied, but  the integer root $a$ of $\Psi_U(X)$ is $-1$.

Throughout this paper we use $p$ and $q$ with or without subscripts for prime numbers. We also use the Landau symbols $O$ and $o$ and the Vinogradov symbols $\gg$ and $\ll$ with their usual meanings. For a set ${\mathcal A}$ of positive integers  and a positive real number $x$ we write ${\mathcal A}(x)={\mathcal A}\cap [1,x)$.

\section{Preliminary results}

As we said, our method closely follows \cite{AL}. However, there are differences. An important ingredient in \cite{ABL} was played by the {\it order of appearance} in the Fibonacci sequence. 
For a fixed $n$, this is denoted by $z(n)$ and is defined as the smallest positive integer $k$ such that $n\mid F_k$. For a prime $p$, $z(p)$ is a divisor of $p-1$ or of $p+1$ according to whether
$p$ is a quadratic residue modulo $5$ or not, except for $p=5$ for which $z(5)=5$. Further, Lemma 1 in \cite{AL} shows that the set of primes $\{p: z(p)<y\}$ is of order of magnitude $O(y^2/\log y)$. In turn, 
this result was used together with a result of Ford from \cite{KFord} (current Lemma \ref{lemma:KFord}) in order 
to ensure that most primes $p$ have $z(p)$ much larger than ${\sqrt{p}}$. This in turn was used together with a result of Shparlinski 
from \cite{Sh} to argue that for such $p$, a set of asymptotic density $1/2$ of all the positive integers $m$ has the property that $F_m$ is a quadratic residue modulo $p$ while the numbers $m$ from the remaining set of asymptotic density $1/2$ have the property that $F_m$ is not a quadratic residue modulo $p$. In the process, we also needed to eliminate numbers $m$ such that $F_m$ is a multiple of $p$; that is, multiples of $z(p)$.

\medskip

In this section, we carry out the necessary modifications to the above scheme for the particular case of the sequence ${\bf U}:=\{U_n\}_{n\ge 0}$ satisfying (i), (ii) and (iii).  The main difference with the argument from \cite{AL} 
is that we do not work with {\it all large primes} $p$, but only with large primes $p$ for which the characteristic polynomial $\Psi_U(X)$ of ${\bf U}$  has exactly one root modulo $p$, which is a subset of relative density $1/2$ 
of all the primes because of condition (i) and the Chebotarev density theorem, cf. 
\cite{SL}. What we need about such primes $p$ is that, for most of them, a set of asymptotic density one half of all the positive integers $n$ has the property that $U_n$ is a quadratic residue modulo $p$ and the remaining half of the positive integers $n$ have the property that $U_n$ is not a quadratic residue modulo $p$. This will follow from Shparlinski's result mentioned above provided that the other conditions stated in (i), (ii) and (iii) are fulfilled. Afterwards, the method from \cite{AL} can be applied with minor modifications. 

\medskip

If $\Psi_U(x)$ has distinct roots $\alpha,\beta$ and $\gamma$, by the theory of linear recurrences
we can write
\begin{equation}
\label{eq:Binet}
U_n=c_\alpha \alpha^n+c_\beta \beta^n+c_\gamma \gamma^n\quad {\text{\rm for~all}}\quad n\in {\mathbb Z},
\end{equation}
for some coefficients $c_{\alpha},~c_{\beta},~c_{\gamma}$ in ${\mathbb K}$.  
We put $\Gamma:=\max\{|\alpha|,|\beta|,|\gamma|\}$. 

\medskip

An important result that we use is due to Beukers \cite{Beu}. Recall that a non degenerate linear recurrence ${\bf V}$ is a linear recurrence of integers whose characteristic polynomial has distinct roots 
whose ratios are not roots of unity.

\begin{lemma}
\label{lem:Beu}
Let ${\bf V}:=\{V_n\}_{n\ge 0}$ be a linearly recurrent sequence of order $3$ whose values are rational integers. Then there are at most $6$ values of $n$ such that $V_n=0$.
\end{lemma}

The following result is an analogue of Lemma 2.1 in \cite{AL}. 
For an arbitrary function $f$ satisfying $f(p)\ge 2$, we denote by ${\mathcal P}_{f(p),U}$ the set
\begin{equation}
\label{eq:PyU}
\{p: U_{pm_i}\equiv 0\pmod p~{\text{\rm for}}~m_1<m_2<\cdots<m_7~{\text{\rm and}}~m_7-m_1\le f(p)\}.
\end{equation}

\begin{lemma}
\label{lem:1}
The estimate 
\begin{equation}
\label{eq:countPyU}
\#{\mathcal P}_{y,U}\ll \frac{y^3}{\log y}
\end{equation}
holds for all $y\ge 2$. 
\end{lemma}

\begin{proof}
Let $p\in {\mathcal P}_{y,U}$. Let ${\pi}$ be any prime ideal of ${\mathcal O}_{\mathbb K}$ dividing $p$. We let $i_j=m_j-m_1$ for $j=1,\ldots,7$. Thus, $0=i_1<i_2<\cdots<i_7\le y$. Then $(c_{\alpha} \alpha^{pm_1}, c_{\beta} \beta^{pm_1}, c_{p\gamma} \gamma^{m_1})^T$ is orthogonal to $(\alpha^{p i_j}, \beta^{p i_j}, \gamma^{p i_j})$ for all values of $j=1,\ldots,7$ in the three dimensional vector space over the finite field ${\mathcal O}_{\mathbb K}/\pi$.
For positive integers $r<s$ put
$$
D(r,s)={\text{\rm det}} \left| \begin{matrix}  1 & 1 & 1\\ \alpha^r & \beta^r & \gamma^r\\ \alpha^s & \beta^s & \gamma^s\\ \end{matrix}\right|.
$$ 
In particular, $D(r,s)\equiv 0\pmod \pi$ for all pairs $(r,s)=(pi_j,pi_k)$ and $1\le j\le k\le 7$. By Fermat's Little Theorem, we get that $D(r,s)\equiv 0\pmod \pi$ for all $(r,s)=(i_j,i_k)$ and $1\le j\le k\le 7$.
By Lemma \ref{lem:Beu}, we shall deduce that there exist $r<s\in \{i_2,\ldots,i_7\}$ such that $D(r,s)\ne 0$. 
More precisely, assume say that $D(i_2,s)=0$ for $s\in \{i_3,i_4,i_5,i_6,i_7\}$. Let $(c_1,c_2,c_3)^T\in {\mathbb K}^3$ 
be any nonzero vector orthogonal to both $(1,1,1)$ and $(\alpha^{i_2},\beta^{i_2},\gamma^{i_2})$. Such a vector exists and is unique up to scalar multiplications because the linear map $T: {\mathbb K}^3\mapsto {\mathbb K}^2$ 
of matrix 
$$
\left(\begin{matrix} 1 & 1 & 1\\ \alpha^{i_2} & \beta^{i_2} & \gamma^{i_2}
\end{matrix}
\right)
$$
has rank exactly $2$ in view of condition (iii) and the fact that $i_2>0$, which together imply that $(\alpha^{i_2},\beta^{i_2},\gamma^{i_2})$ is not parallel to $(1,1,1)$.  Since all vectors $(\alpha^{i_j},\beta^{i_j},\gamma^{i_j})$ for $j=1,\ldots,7$ are linear combinations of $(1,1,1)$ and $(\alpha^{i_1},\beta^{i_1},\gamma^{i_1})$, we get that 
$$
c_1\alpha^{i_j}+c_2\beta^{i_j}+c_3\gamma^{i_j}=0\qquad {\text{\rm for~all}}\qquad j=1,\ldots,7.
$$
This means that ${\bf V}:=\{V_n\}_{n\ge 0},$ whose Binet formula is given by
$$
V_n=c_1\alpha^n+c_2\beta^n+c_3\gamma^n,
$$
has the property that $V_n=0$ for $7$ different values of $n$. This contradicts Lemma \ref{lem:Beu}, except that we have to check for the condition that $V_n$ has integer values. Since $(c_1,c_2,c_3)$ is parallel to the cross product of 
$(1,1,1)$ and $(\alpha^{i_2},\beta^{i_2},\gamma^{i_2})$,  we get that  
$$
(c_1,c_2,c_3)=\lambda(\gamma^{i_2}-\beta^{i_2},\alpha^{i_2}-\gamma^{i_2},\beta^{i_2}-\alpha^{i_2})
$$
for some nonzero scalar $\lambda\in {\mathbb K}$. We already know that $V_0=c_1+c_2+c_3=0$. Computing $V_1$ and $V_2$, we get
\begin{eqnarray*}
V_1 & = & \lambda \left(\alpha(\gamma^{i_2}-\beta^{i_2})+\beta(\alpha^{i_2}-\gamma^{i_2})+\gamma(\beta^{i_2}-\alpha^{i_2})\right),\\
V_2 & = & \lambda\left(\alpha^2 (\gamma^{i_2}-\beta^{i_2})+\beta^2(\alpha^{i_2}-\gamma^{i_2})+\gamma^2(\beta^{i_2}-\alpha^{i_2})\right).
\end{eqnarray*}
Looking at the expressions multiplied by $\lambda$ in the right--hand side above, we see that the permutations $(123)$ and its square leave both $V_1$ and $V_2$ unchanged, whereas the transpositions $(12), (23), (13)$ change $V_1$ and $V_2$ to their negatives. This shows that putting $\Delta$ for the discriminant of ${\mathbb K}$, we get that 
both in the case when ${\mathbb K}$ has degree $6$ and $G=S_3$, as well as in the case when ${\mathbb K}$ has degree $2$ and $G={\mathbb Z}/2{\mathbb Z}$, we have that $V_n{\sqrt{\Delta}}$ is an integer for $n=0,~1,~2$. Hence, by induction on $n$ using the third order linear recurrence for 
${\bf V}$, we get that $V_n{\sqrt{\Delta}}$ is an integer for all $n\ge 0$, so Lemma \ref{lem:Beu} 
(due to Beukers) indeed applies and tells that we cannot have $V_n=0$ for $7$ values of $n$.

 It then follows that there exist $r<s$ in $\{i_2,\ldots,i_7\}$ such that $D(r,s)\ne 0$. But $\pi\mid D(r,s)$.  Further, notice that $D(r,s)^2$ 
is an integer since it is obviously an algebraic integer and any conjugation from ${\mathbb K}$ just permutes the columns of the determinant whose value is $D(r,s)$, therefore it will not change the square of it. Thus, $p\mid D(r,s)^2$. Hence,
$$
\prod_{p\in {\mathcal P}_{y,U}} p\mid \prod_{\substack{0<r<s\le y\\ D(r,s)\ne 0}} D(r,s)^2.
$$
Since clearly $|D(r,s)|\ll \Gamma^{2s}$, we get that
$$
\prod_{p\in {\mathcal P}_{y,U}} p\ll \prod_{1\le r<s\le y} \Gamma^{4s}<\Gamma^{4y^3},
$$
so 
\begin{equation}
\label{eq:sum}
\sum_{p\in {\mathcal P}_{y,U} }\log p\ll y^3.
\end{equation}
Put $t:=\#{\mathcal P}_{y,U}$ and denote by $p_1,p_2,\ldots$ all the consecutive
primes. By the prime number theorem (or Chebyshev's estimates), we have
$$
\sum_{p\in {\mathcal P}_{y,U} }\log p \ge \sum_{p\le p_t}\log p \gg p_t \gg t \log t,
$$ 
and hence
$t\log t\ll y^3$,
which implies the desired estimate \eqref{eq:countPyU}.
\end{proof}

For an integer $n$ denote by $P(n)$ the largest prime factor of $n$ with the convention that $P(0)=P(\pm 1)=1$. Given a positive real number $y$, a positive integer $n$ is called $y$-{\it smooth} if $P(n)\le y$. We need the following well-known bound from the theory of smooth numbers. Put
$$
\Psi(x,y)=\#\{n\le x: P(n)\le y\}.
$$
The following is Theorem 1 of Chapter III.5 in \cite{Ten}.

\begin{lemma}
\label{lem:2}
The estimate 
$$
\Psi(x,y)\ll x\exp(-u/2)
$$
holds for all $x\ge y\ge 2$ with $u=\log x/\log y$.
\end{lemma}

Better (sharper) bounds for $\Psi(x,y)$ hold when $y$ is not too small with respect to $x$ (see, for example, the corollary to Theorem 3.1 in \cite{CEP}).

We shall need some information concerning the number of divisors of shifted primes which are in  a given interval. Namely, let
$$
H(x,y,z)= \#\{ n\le x\;:\; d\mid n\, {\text{\rm for~some}}\, d \in (y,z) \},
$$  
and for a given non-zero integer $\lambda$ put
$$
P(x,y,z;\lambda)=\#\{ p\le x\;:\; d\mid p+\lambda\, {\text{\rm for~some}}\, d\in (y,z) \}.
$$
The following result appears as Theorem 6 in~\cite{KFord}.

\begin{lemma}
\label{lemma:KFord}
If $100 \le y \le x^{1/2},$ and $2y \le z \le y^2,$ then
$$
H(x,y,z) \asymp x u^{\delta}({\log(2/u)})^{-3/2},
$$
where $u$ is defined implicitly by $z=y^{1+u}$ and 
$$
\delta = 1 - \frac{1+ \log\log 2}{\log 2}= 0.086071\ldots \, .
$$
Furthermore,   let $1 \le y \le x^{1/2},$ and $y + (\log y)^{2/3} \le z \le x.$ The following estimate holds
$$
P(x,y,z;\lambda) \ll_{\lambda} \frac{H(x,y,z)}{\log x}. \qquad \qed
$$
\end{lemma}

We shall only need Lemma \ref{lemma:KFord} for $\lambda\in \{\pm 1\}$.

Now we shall introduce a special set of primes which is important for our arguments. We let
\begin{equation}
\label{eq:zet}
{\mathcal Z}=\{p: \Psi_U(x){\text{\rm ~has~exactly~one~root~modulo~}p}\}.
\end{equation}
If $p$ is sufficiently large and is in ${\mathcal Z}$, then its Frobenius, regarded as an element of $G,$ is in the conjugacy class of the transpositions $\{(12),(23),(13)\}$ when $G=S_3$ and is 
the only nonidentical element of $G$ (which is a transposition of $S_3$) when $\Psi_U(X)$ has an integer root $a$ and ${\mathbb K}$ is quadratic. So, in either case, the Frobenius of such a $p$ is in a conjugacy class 
of index $2$ in $G$. 
It is now an immediate consequence of the Chebotarev Density Theorem, cf. 
\cite{SL}, that ${\mathcal Z}$ contains asymptotically half of the primes, that is, 
$$
\#{\mathcal Z}(x)=\left(1+o(1)\right)\frac{x}{2\log x}\quad {\text{\rm as}}\quad x\to\infty.
$$
For lack of a better notation, we write $\alpha$ for the unique root of $\Psi_U(X)$ modulo $p$ and put $\beta$ and $\gamma$ for the remaining two roots of $\Psi_U(X)$. In case $\Psi_U(X)$ has an integer root $a$, then certainly 
$\alpha=a$. Modulo $p$, we have  
$$
\alpha^p\equiv \alpha\pmod p,\quad \beta^p\equiv \gamma\pmod p,\quad \gamma^p\equiv \beta\pmod p.
$$
Thus, if $n=pm$, where $p\in {\mathcal Z}$, then using \eqref{eq:Binet}, we get on putting
$$V_m=c_{\alpha} \alpha^m+c_{\beta} \gamma^m+c_{\gamma} \beta^m,\qquad {\text{\rm that}}$$
\begin{equation}
\label{eq:reduce}
U_n\equiv c_{\alpha} (\alpha^p)^m+c_{\beta} (\beta^p)^m+c_{\gamma}(\gamma^p)^m\equiv V_m\pmod p.
\end{equation}
Note that ${\bf V}:=\{V_m\}_{m\ge 0}$ is a linearly recurrent sequence satisfying the same recurrence relation as ${\bf U}$ but it is defined only modulo $p$. The above formula is the analogue of Lemma 2.5 in \cite{AL}. Next we need to understand the periods of 
${\bf U}$ and ${\bf V}$ modulo $p$ for $p\in {\mathcal Z}$. 
\begin{defi}
The period $t(p)$ is the smallest positive integer $k$ such that $U_n\equiv U_{n+k}\pmod p$ for all $n\ge 0$ (or, 
$V_m\equiv V_{m+k}\pmod p$ for all $m\ge 0$, respectively). Let $k:=k(p)$ be the minimal positive
integer such that all three congruences
\begin{equation}
\label{eq:orders}
\alpha^k\equiv1 \pmod p,\quad \beta^k\equiv 1\pmod p,\quad \gamma^k\equiv 1\pmod p
\end{equation}
hold. 
\end{defi}
\noindent Note that $k(p)$ is a period of  ${\bf U}$ and ${\bf V}$. Hence, $t(p)$ divides $k(p)$. For large $p$, we 
have in fact that $t(p)=k(p)$. Since we need a precise form of this statement including a precise way to quantify ``all sufficiently large $p$" for the case where we only work with the 
subsequence $\{U_{c+dn}\}_{n\ge 0}$ of ${\bf U}$, we record such a statement below.
\begin{lemma}
\label{lem:periodd}
For each positive integer $d$ and uniformly in $c\in \{0,1,\ldots,d-1\}$, the number of primes $p$ such 
that the period of $\{U_{c+dn}\}_{n\ge 0}$ modulo $p$ is not the smallest positive integer $k:=k(p,d)$ with 
\begin{equation}
\label{eq:congruences}
\alpha^{dk}\equiv 1\pmod p,\quad \beta^{dk}\equiv 1\pmod p\quad {\text{ and}}\quad \gamma^{dk}\equiv 1\pmod p
\end{equation}
is $O(d/\log d)$, where the implied constant depends at most on ${\bf U}$. 
We have $k(p,d)=k(p)/{\text{\rm gcd}}(k(p),d)$.
\end{lemma}
\begin{proof} 
We assume $p$ is sufficiently large so that it does not divide the discriminant of $\Psi_U(X)$ and the numbers $c_{\alpha},~c_{\beta}$ and $c_{\gamma}$ are defined and nonzero modulo any prime ideal $\pi$ of ${\mathcal O}_{\mathbb K}$ dividing $p$. 
Clearly, the period is the smallest $k$ such that $U_{c+d(n+k)}\equiv U_{c+dn}\pmod p$ for $n=0,1,\ldots.$ Writing the above congruences down using the Binet formulas we get that
$$
c_{\alpha} \alpha^c(\alpha^{dk}-1) \alpha^{dn}+c_{\beta} \beta^c (\beta^{dh}-1) \beta^{dn}+c_{\gamma} \gamma^c (\gamma^{dk}-1)\gamma^{dn}\equiv 0\pmod \pi.
$$
Hence, the vector $(c_{\alpha}\alpha^c (\alpha^{dk}-1)),c_{\beta}\beta^{c} (\beta^{dk}-1),c_{\gamma}\gamma^{c} (\gamma^{dk}-1))^T$ is orthogonal to $(\alpha^{dn}, \beta^{dn},\gamma^{dn})$ in the finite field ${\mathcal O}_{\mathbb K}/\pi$ of characteristic $p$. 
We need to bound the number of primes $p$ such that that the above vector is not the zero vector. To do
so note that if the above vector is not the zero vector, then taking $n=0,~r,~s$, we get that 
$D(rd,sd)\equiv 0\pmod \pi$. By the argument from Lemma \ref{lem:1}, there exist $1\le r<s\le 6$ such that $D(rd,sd)\ne 0$. 
Hence, using Lemma \ref{lem:Beu} (due to Beukers), we find that $p\mid D(rd,sd)^2$. So, the primes $p$ for which  one of the congruences \eqref{eq:congruences} fails 
must divide the nonzero integer
$$
\prod_{\substack{1\le r<s\le 6\\ D(rd,sd)\ne 0}} D(rd,sd)^2
$$
and therefore their product also divides the above nonzero integer. The size of the above integer is at most $\Gamma^{O(d)}$. Hence, the number of such primes is $O(d/\log d)$ by the argument from the conclusion of 
the proof of Lemma \ref{lem:1}. 

The final assertion is obvious.
\end{proof}

 Let ${\text{\rm ord}}_p(\bullet)$ denote the order of $p$ function defined either on ${\mathbb Z}/p{\mathbb Z}$ or on some finite extension of it.

\begin{lemma}
\label{lem:5}
Assume that $p\in {\mathcal Z}$ is sufficiently large. Then
\begin{itemize}
\item[(i)] ${\text{\rm ord}}_p(\alpha)\mid p-1$ and ${\text{\rm ord}}_p(\beta/\gamma)\mid p+1$.
\item[(ii)] Let $k(p)$ and $t(p)$ be as in Definition \ref{eq:orders} and $t(p)=k(p)$, then 
\begin{equation}
\label{eq:divisibilities}
{\text{\rm ord}}_p(\alpha){\text{\rm ord}}_p(\beta/\gamma)\mid 2t(p)\mid 8 {\text{\rm ord}}_p(\alpha){\text{\rm ord}}_p(\beta/\gamma)\mid 8(p-1)(p+1).
\end{equation}
\end{itemize}
\end{lemma}

\begin{proof}
(i) We will deal only with the case $\alpha \beta\gamma=1$, as the argument in case $\alpha \beta\gamma=-1$ is similar. 
Since $\alpha^p\equiv \alpha\pmod p$, we have that ${\text{\rm ord}}_p(\alpha)\mid p-1$. Since $\beta^p\equiv \gamma\pmod p$, it follows that $\beta^{p+1}\equiv \beta\gamma\equiv \alpha^{-1}\pmod p$ and the same
conclusion is reached with $\beta$ replaced by $\gamma$. Thus, $\beta^{p+1}\equiv \gamma^{p+1}\pmod p$,  or $(\beta/\gamma)^{p+1}\equiv 1\pmod p$. Thus, ${\text{\rm ord}}_p(\beta/\gamma)\mid p+1$. This finishes (i).

\medskip

(ii) Let $L={\text{\rm lcm}}[{\text{\rm ord}}_p(\alpha), {\text{\rm ord}}_p(\beta/\gamma)]$. Since 
by assumption $t(p)=k(p)$, it follows that $\alpha^L\equiv 1\pmod p$ and $(\beta/\gamma)^L\equiv 1\pmod p$. 
In particular, $t(p)$ is a multiple of $L$. Now assume that $\alpha^L\equiv 1\pmod p$ and $(\beta/\gamma)^L\equiv 1\pmod p$. Thus, $\beta^L\equiv \gamma^L\pmod p$. If $\alpha\beta\gamma=\pm 1$, then 
$$
1\equiv (\alpha \beta\gamma)^{2L}\pmod p\equiv \beta^{4L}\pmod p.
$$
If $\beta\gamma=\pm 1$, then 
$$
1=(\beta\gamma)^{2L}\pmod p=\beta^{4L}\pmod p.
$$
Hence, in either case $\beta^{4L}=1$. Similarly, we have $\gamma^{4L}\equiv 1\pmod p$ in these two cases.
This shows that $t(p)\mid 4L$. We now only need to understand the relation between $L$ and the product of ${\text{\rm ord}}_p(\alpha)$ and ${\text{\rm ord}}_p(\beta/\gamma)$. By (i), we have that ${\text{\rm ord}}_p(\alpha)\mid p-1$ and  ${\text{\rm ord}}_p(\beta/\gamma)\mid p+1$. Since $\gcd(p-1,p+1)=2$, we get that $L={\text{\rm ord}}_p(\alpha){\text{\rm ord}}_p(\beta/\gamma)/D$, where $D\in \{1,2\}$. Now (ii) is clear. 
\end{proof}

In view of the above results, we introduce two other sets of primes which are similar to the sets ${\mathcal P}_{y,U}$ defined before Lemma 2.1 in \cite{AL}. Namely, for a real number $y\ge 2$ let
\begin{eqnarray*}
{\mathcal K}_y & = & \{p: p\mid N_{\mathbb K/\mathbb Q}(\alpha^k-1)~{\text{\rm for~some}}~k\le y\},\\
{\mathcal L}_y & = & \{p: p\mid N_{\mathbb K/\mathbb Q}((\beta/\gamma)^k-1)~{\text{\rm for~some}}~k\le y\},
\end{eqnarray*}
where $N$ denotes the norm.
The following result can be proved in the same way as Lemma 2.1 in \cite{AL}. 
Suppose $\Psi_U$ satisfies condition (ii) of Theorem \ref{thm:1} and has an integer root $a$. The reason that
we need to ensure that $a\ne \pm 1$
is related to the proof of the result below (especially the estimate for $\#{\mathcal K}_y$), since of course if $a=\pm 1$ and $y\ge 2$, then ${\mathcal K}_y$ contains all the primes.
\begin{lemma}
\label{lem:QR}
We have 
$$
{\rm max}\{\#{\mathcal K}_y,\#{\mathcal L}_y\}\ll \frac{y^2}{\log y}.
$$
\end{lemma}

Recall that a ``multiplier" modulo $p$ is a residue class $\lambda$ modulo $p$ such that for some $n$ we have $\alpha^n\equiv \beta^n \equiv \gamma^n\equiv \lambda\pmod p$. The multipliers form a multiplicative group in ${\mathbb Z}/p{\mathbb Z}$. 
In our case, recall that either $\Psi_U(X)$ is irreducible over ${\mathbb Q}$, in which case $a_3=\pm 1$, or $\Psi_U(X)$ has a root $a$ and the product of the other two roots is $\pm 1$. 
Hence, either $\alpha\beta\gamma=\pm 1$, or $\beta\gamma=\pm 1$. We thus get that any multiplier $\lambda$ satisfies $\lambda^3\equiv (\alpha\beta\gamma)^n\equiv \pm 1\pmod p$ in the first case, and $\lambda^2\equiv (\beta\gamma)^n\equiv \pm 1\pmod p$ in the second case, so the group of multipliers has at most $6$ elements. 

Given an arithmetic progression $c\pmod d$, we 
denote by $t_{c,d,p}$ the period of the sequence $(V_{c+dn})_{n\ge 0}$ modulo $p$. By Lemma \ref{lem:periodd} we have the equality $t_{c,d,p}=t(p)/{\text{\rm gcd}}(d,t(p))$, 
except for a set of primes $p$ of cardinality $O(d/\log d)$,  
 We record this as the first part of the next lemma. The second part of it follows from the bound on in \cite[p. 86]{ESPW} and is based on results from \cite{Sh}.
\begin{lemma}
\label{lem:period}
Assume that $c\ge 0,~d>0$ are integers. Then 
\begin{itemize}
\item[(i)] $t_{c,d,p}=t(p)/{\text{\rm gcd}}(d,t(p))$ with $O(d/\log d)$ exceptions. 
\item[(ii)] We have 
$$
\sum_{k=1}^{t_{c,d,p}} \left(\frac{V_{c+dk}}{p}\right)\ll p.
$$
\end{itemize}
\end{lemma}
In fact the result given in \cite[p. 86]{ESPW} together with our remark that the group of multipliers for $\{V_m\}_{m\ge 0}$ has at most $6$ elements shows that the implied constant in the above Vinogradov symbol $\ll$ can be taken to be $6$.

\section{The proof of Theorem \ref{thm:1}}

We begin by discarding several subsets of integers $n\in [1,x)$ which 
on removal make our problem easier to deal with.
We 
proceed along the lines of \cite{AL} with the same choice of parameters
so we will only make the arguments explicit in case there are new ideas involved.

Recall the definition \eqref{muuuh} of ${\mathcal M}_U$.
To simplify notation we omit the subscript $U$ on ${\mathcal M}_U$ and 
just write ${\mathcal M}$. We write ${\mathcal M}_1,~{\mathcal M}_2$ and so on for subsets of ${\mathcal M}$.  
Let $x$ be a large positive real number. Put $y_1 = \exp(\log x /\log\log  x)$. 
Let
$$
{\mathcal M}_1(x)=\{n\le x:P(n)\le y_1\}.
$$
By Lemma \ref{lem:2}, we have
\begin{equation}
\label{eq:N1bound}
\#{\mathcal M}_1(x)=\Psi(x,y_1) \ll x\exp(-u/2)=\frac{x}{(\log x)^{1/2}}.
\end{equation}
Here, $u=\log x/\log y_1$. Next let $z_1=(\log x)^3$. Let $\kappa\in (0,1)$ to be fixed later. Put
$$
\cM_2(x)=\{n\le x: p^2\mid n~{\text{\rm for~some~prime}}~p>z_1^{\kappa}\}.
$$
Note that
\begin{equation}
\label{eq:N2bound}
\#{\mathcal M}_2(x)\le \sum_{z_1^{\kappa}\le p\le x^{1/2}} \frac{x}{p^2}\le x\sum_{m\ge z_1^{\kappa}} \frac{1}{m^2}\ll \frac{x}{z_1^{\kappa}}=\frac{x}{(\log x)^{3\kappa}}.
\end{equation}
Next we let 
$$
\cP = {\mathcal Z}\cap {\mathcal P}_{p^{1/4},U},
$$ 
where the set ${\mathcal Z}$ is defined in 
\eqref{eq:zet} and ${\mathcal P}_{p^{1/4},U}$ in \eqref{eq:PyU}. 
By Lemma \ref{lem:1}, we know that 
$$
\#\cP(x) \le \#{\mathcal P}_{x^{1/4},U}\ll \frac{x^{3/4}}{\log x}.
$$
Put
$$
\cM_3(x)=\{n\le x: p\mid n~{\text{\rm for~ some}}~p\in {\mathcal P}~{\text{\rm with}}~p\ge z_1\}.
$$
The number of $n\le x$ which are multiples of $p$ is $\lfloor x/p\rfloor\le x/p$. Summing up over all the possibilities for 
$p$, we get, by partial summation,  
\begin{eqnarray}
\label{eq:N3bound}
\#{\mathcal M}_3(x) & \le & \sum_{\substack{z_1\le p\le x\\ p\in \cP}}\frac{x}{p} = x \int_{z_1}^x\frac{\mathrm{d}\#\cP(t)}{t} 
 =x\left(\frac{\#\cP(t)}{t}\Big|_{z_1}^x+ \int_{z_1}^x \frac{\#\cP(t)}{t^2}  {\mathrm{d}}t\right)\nonumber\\
& \ll  & x\left(\frac{\#\cP(x)}{x} + \int_{z_1}^{x}\frac{{\mathrm{d}}t}{t^{5/4}}\right) \ll x\left(\frac{1}{x^{1/4}\log x}+\left(-\frac{4}{t^{1/4}}\Big|_{t=z_1}^{t=x}\right)\right)\nonumber\\
& \ll & \frac{x}{(\log x)^{3/4}}.
\end{eqnarray} 
Next we define the set 
$$
\cM_4(x)=\{n\le x: n\not\in {\mathcal M}_3(x),~ p\mid \gcd(n,U_n)~{\text{\rm for~some}}~p\in {\mathcal Z}~{\text{\rm with}}~p>z_1\}.
$$
If $n\in {\mathcal M}_4(x)$, then $p\mid \gcd(n,U_n)$ for some prime $p\in {\mathcal Z}$ with $p>z_1$. Write $n=pm$. Further we have, $U_n=U_{pm}\equiv 0\pmod p$. 
Since $p\not\in {\mathcal P}_{p^{1/4}, U}$, it follows that each interval of the form $[1+p^{1/4}\ell,p^{1/4}(\ell+1)]$
contains at most $6$ of the $m_i$'s for all integers $\ell\ge 0$. The $m_i$'s for which $\ell= 0$ (so $m_i\le p^{1/4}$) give us a total of at most $6\pi(x)=O(x/\log x)$ possibilities for $n$. The remaining ones give a total of at most
$$
\sum_{\substack{p\in {\mathcal Z}\backslash {\mathcal P}\\ p>z_1}}\frac{6x}{p^{5/4}} \ll x\sum_{p>z_1} \frac{1}{p^{5/4}}\ll \frac{x}{(\log x)^{3/4}}
$$ possibilities.
Hence,
\begin{equation}
\label{eq:N4bound}
\#{\mathcal M}_4(x)\ll \frac{x}{(\log x)^{3/4}}.
\end{equation}
Assume for the moment that $n\le x$ is in ${\mathcal M}(x)\backslash \bigcup_{i=1}^4 {\mathcal M}_i(x)$. Thus 
we can write
\begin{equation}
\label{eq:maineq}
U_n=u^2+nv^2
\end{equation}
for some integers $u$ and $v$ (depending on $n$). For large $x$ we have $y_1>z_1$, so, since $n\not\in {\mathcal M}_3(x)$, there is some prime $p>z_1$ 
such that $p\mid n$. Since $n\not\in {\mathcal M}_2(x)\cup {\mathcal M}_4(x)$, it follows that $p\| n$ and $p\nmid U_n$.  Assume for now that $p\in {\mathcal Z}$. Then writing $n=pm$, we have $\gcd(m,p)=1$ and $U_n\equiv V_m\pmod p$ by \eqref{eq:reduce}.
Reducing equation \eqref{eq:maineq} modulo $p$, we get
$$
U_n\equiv u^2\pmod p
$$
and $p\nmid u$. Thus, 
$$
1=\left(\frac{U_n}{p}\right)=\left(\frac{V_m}{p}\right).
$$
We conclude from this argument that whenever we have a representation of $n$ of the form $n=mp$, with $p>z_1$, then $V_m$ is a quadratic residue modulo $p$. In order to use this information efficiently, we remove some more integers $n\le x$. Let
\begin{eqnarray*}
\cM_5(x)  =  \{n\le x &:& n\not\in {\mathcal M}_2(x)~ {\text{\rm and there is}}~q>z_1^{\kappa},~q\mid \gcd(k(p_1),k(p_2))\nonumber\\
&& {\text{\rm for}}~p_1\ne p_2~{\text{\rm both~in}}~{\mathcal Z}~{\text{\rm with}}~p_1p_2\mid n,\nonumber\\
&& {\text{\rm or}}~q\mid \gcd(n,k(p))~{\text{\rm for~some}}~p\mid n~{\text{\rm with}}~p\in {\mathcal Z}\}.
\end{eqnarray*}
Assume that $n\in {\mathcal M}_5(x)$.
Observe that if $x$ is large, then $q$ is large, so the condition $q \mid k(p)$ together with the condition that $p\in {\mathcal Z}$ implies $q \mid p \pm 1$. Hence, either $q\mid n$ and $q\mid p\pm 1$ for some $p>z_1$ with $p\in {\mathcal Z}$, or 
there are $p_1,~p_2$ both in ${\mathcal Z}$ dividing $n$ such that $q\mid \gcd(p_1\pm 1,p_2\pm 1)$. All this follows from Lemma \ref{lem:5}. In either case, the argument from \cite{AL} applies and gives 
\begin{equation}
\label{eq:N5bound}
\#{\mathcal M}_{5}(x) \ll  \frac{x(\log\log x)^2}{(\log x)^{3\kappa}}.
\end{equation}

For a prime $p$ write 
$k(p)=a_p b_p$, where $P(a_p)\le (\log p)^3$ and $b_p$ has only prime factors larger than $(\log p)^3$. 
Let $z_2=\exp(18(\log\log x)^2)$ and 
put
$$
\#\cM_6(x) = \{n\le x: a_p > z_2~ {\text{\rm for some prime }}~p \mid n\}.
$$
The argument from \cite{AL} applies and gives
\begin{equation}
\label{eq:N6bound}
\#{\mathcal M}_6(x)\ll x\log x\left(\frac{1}{(\log x)^3}+\frac{1}{(\log x)^3} \int_{z_2}^x \frac{{\mathrm{d}}t}{t}\right)\ll \frac{x}{\log x}.
\end{equation}
Next let $z_3=\exp((\log x)^{\kappa})$. We next will discard positive integers $n$ having a prime factor $p>z_3$ for which $k(p)$ is ``small" in a sense that will be made more precise below. Put $c=20\kappa^{-2}$ and
define the following sets of primes 
\begin{eqnarray*}
{\mathcal Q}_1 & = & {\mathcal Z} \cap {\mathcal K}_{p^{1/2}/\log p};\qquad {\mathcal R}_1  =  {\mathcal Z} \cap {\mathcal L}_{p^{1/2}/\log p};\nonumber\\
{\mathcal Q}_2 & = & \left\{p: d\mid p-1~{\text{\rm and}}~ d\in \left[\frac{p^{1/2}}{\log p}<z(p)<p^{1/2} \exp\left(c(\log\log p)^2\right)\right]\right\};\nonumber\\
{\mathcal R}_2 & = & \left\{p: d\mid p+1~{\text{\rm and}}~d\in \left[\frac{p^{1/2}}{\log p}<z(p)<p^{1/2} \exp\left(c(\log\log p)^2\right)\right]\right\}.\nonumber
\end{eqnarray*}
We need estimates for the counting functions of $\#{\mathcal Q}_1(t)$, $\#{\mathcal Q}_2(t)$, $\#{\mathcal R}_1(t)$ and $\#{\mathcal R}_2(t)$. For $\#{\mathcal Q}_1(t)$, we have
\begin{equation}
\label{eq:Q1}
\#{\mathcal Q}_1(t)\le \#{\mathcal Q}_{t^{1/2}/\log t}\ll \frac{t}{(\log t)^3},
\end{equation}
by Lemma \ref{lem:QR} with $y=t^{1/2}/\log t$. A similar inequality holds for $\#{\mathcal R}_1(t)$. For $\#{\mathcal Q}_2(t)$, we first deal with ${\mathcal Q}_2\cap [t/2,t]$. Let $p$ be a prime in ${\mathcal Q}_2\cap [t/2,t]$
and $t$ be large. Then
\begin{eqnarray*}
 \frac{p^{1/2}}{\log p} & > & \frac{t^{1/2}}{2^{1/2}\log(t/2)}>\frac{t^{1/2}}{2\log t};\\
 p^{1/2} \exp\left(c(\log\log p)^2\right) & < &  t^{1/2} \exp\left(c(\log\log t)^2\right).
\end{eqnarray*}
It follows from this that if $p$ is in ${\mathcal Q}_2$, then $p-1$ 
has a divisor in the interval $(y,z)$, where $y=t^{1/2}/(2\log t)$ and $z=t^{1/2}\exp\left(c(\log\log t)^2\right)$. 
The argument from \cite{AL} based on estimates from \cite{KFord} gives
\begin{equation}
\label{eq:Q2}
\#({\mathcal Q}_2\cap [t/2,t])\le \sum_{\lambda\in \{\pm 1\}} P(t,y,z,\lambda)\ll \frac{H(t,y,z)}{\log t}\ll \frac{t}{(\log t)^{1+\delta}},
\end{equation}
with $\delta$ as in Lemma \ref{lemma:KFord}.
Replacing $t$ by $t/2$, then by $t/4$, etc. and summing the above estimates \eqref{eq:Q2}, we get
\begin{equation}
\label{eq:Q2t}
\#{\mathcal Q}_2(t)\ll \frac{t}{(\log t)^{1+\delta}},
\end{equation}
as in \cite{AL}. A similar argument holds with ${\mathcal Q}_2$ replaced by ${\mathcal R}_2$ (just change $p-1$ to $p+1$). 
Comparing estimate \eqref{eq:Q2t} with \eqref{eq:Q1}, we get that if we put ${\mathcal Q}_3={\mathcal Q}_1\cup {\mathcal Q}_2$ and ${\mathcal R}_3={\mathcal R}_1\cup {\mathcal R}_2$, then 
$$
\#{\mathcal Q}_3(t)\le \#{\mathcal Q}_1(t)+\#{\mathcal Q}_2(t)\ll \frac{t}{(\log t)^{1+\delta}},
$$
and a similar estimate holds with ${\mathcal Q}$ replaced by ${\mathcal R}$.
Next we consider
\begin{equation*}
{\mathcal M}_7(x)=\{n\le x:{\text{\rm there~exists}}~p>z_3,~p\in {\mathcal Z},~p\mid n,~p\in {\mathcal Q}_3\cup {\mathcal R}_3\}.
\end{equation*}
The argument from \cite{AL} now applies and gives
\begin{equation}
\label{eq:N7bound}
\#{\mathcal M}_7(x) \ll \frac{x}{(\log x)^{\kappa\delta}}.
\end{equation}
We next fix $\lambda\in (0,(1-\kappa)/2)$ to be determined later, put $K=\lfloor \lambda\log\log x\rfloor$, $y_2=\exp(\log x/(\log\log x)^2)$, ${\mathcal I}=(z_3,y_2)$
and
$$
\omega_{{\mathcal I}\cap {\mathcal Z}}(n)=\sum_{\substack{p\in {\mathcal I}\cap {\mathcal Z}\\ p\mid n}} 1.
$$
Let
$$
{\mathcal M}_8(x)=\{n\le x: n\not\in {\mathcal M}_2(x), \omega_{{\mathcal I}\cap {\mathcal Z}}(n)< K\}.
$$ 
We now follow \cite{AL}. The only difference is in the estimate of the sum
$$
S=\sum_{p\in {\mathcal I}\cap {\mathcal Z}} \frac{1}{p}=\left(\frac{1-\kappa}{2}\right)\log\log x-\log\log\log x+O(1),
$$
which is by a factor of $1/2$ smaller than the analogous sum $S$ in \cite{AL}. The presence of the factor $1/2$ is due to the fact that we only work with primes $p\in {\mathcal Z}$, a subset of 
relative density $1/2$ in the set of all primes. Following \cite{AL} and/or invoking Theorems 08 and 09 in \cite{HT}, we get
\begin{equation}
\label{eq:N8bound}
\#{\mathcal M}_8(x)  \ll  \frac{x(\log\log x)^{O(1)}}{(\log x)^{\mu}},
\end{equation}
where 
$$
\mu=\left(\frac{1-\kappa}{2}\right)-\lambda\log\left(\frac{e(1-\kappa)}{2\lambda}\right).
$$
Let $n\not\in \bigcup_{i=1}^8 {\mathcal M}_i(x)$. Write $n=Pm$, where $P=P(n)$.  Fix $m$. The main idea is that now $n$ has $K$ representations of the form
$n=p_im_i$, where $p_i\in {\mathcal I}\cap {\mathcal Z}$. Say $n=p_i m_i$, where $p_1<p_2<\cdots<p_K$ are the first (smallest) prime factors of $n$ in ${\mathcal I}\cap {\mathcal Z}$ which exist because $n\not\in {\mathcal M}_8(x)$. As in \cite{AL}, cf. the first sentence following (\ref{eq:N5bound}) in the present paper, we write 
$k(p_i)=a_{p_i} b_{p_i}$ for $i=1,\ldots,K$ and conclude that $\gcd(b_{p_i},b_{p_j})=1$ for $i\ne j$ both in $\{1,2,\ldots,K\}$ because $n\not\in {\mathcal M}_5(x)\cup {\mathcal M}_6(x)$. Further, by Lemma \ref{lem:5} and using the fact that the
$p_i$ are all sufficiently large for $x$ sufficiently large, we get that the equalities $t(p_i)=k(p_i)$ hold for all $i=1,\ldots,K$. Hence,
$k(p_i)=\delta(p_i) {\text{\rm ord}}_{p_i}(\alpha) {\text{\rm ord}}_{p_i}(\beta/\gamma)$, where $\delta(p_i)\in \{1/2,1,2\}$. We keep the notations from \cite{AL}, where 
$$
U(m)={\text{\rm lcm}}[t(p_1),\ldots,t(p_K)]\quad {\text{\rm and}}\quad V(m)={\text{\rm lcm}}[a_{p_1},\ldots,a_{p_K}],
$$
and hope that the reader will not confuse these notations with $U_m$, or $V_m$, respectively. 
We then get that $V(m)$ is ``small", namely 
$$
V(m)\le 4\exp(18\lambda(\log\log x)^3).
$$ 
We shall work with $\{V_{c+dn}\}_{n\ge 0}$ where $d=V(m)$ and $0\le c<d$ will be be chosen later. There is a further nuisance here which  was not present in the problem treated in \cite{AL}, in that it might be possible that $t_{c,d,p}\ne t(p)/{\text{\rm gcd}}(t(p),d)$ for one of the primes $p=p_1,\ldots,p_K$ which we are working with. But Lemma \ref{lem:period} tells that, for each fixed $d$, the number of such primes $p$ is at most $O(d)$. Put  
$z_4=\exp(18 (\log\log x)^3)$ and note that $V(m)<z_4$ for all sufficiently large $x$. Put ${\mathcal Q}_4$ for the set of primes $p>z_3$ such that $t_{c,d,p}\ne t(p)/{\text{\rm gcd}}(t(p),d)$ for some pair $(c,d)$ with $d<z_4$ and $c\in \{0,1,\ldots,d-1\}$. Then 
$$
\#{\mathcal Q}_4\ll \sum_{d\le z_4} d \ll z_4^2.
$$
So, letting
$$
{\mathcal M}_{9} (x)=\{n\le x:p\mid n~{\text{\rm for~some}}~p>z_3~{\text{\rm and}}~p\in {\mathcal Q}_4\},
$$
we get that 
\begin{equation}
\label{eq:N10bound}
\#{\mathcal M}_{9}(x)\le \sum_{p\in {\mathcal Q}_4} \frac{x}{p}\ll \frac{x\#{\mathcal Q}_4}{z_3}\ll \frac{x z_4^2}{z_3}\ll \frac{x}{\log x}.
\end{equation}
From now on, we work in ${\mathcal M}_{10}(x)={\mathcal M}(x)\backslash \left(\bigcup_{i=1}^9 {\mathcal M}_i(x)\right)$.  We also fix the residue class $c$ of $P$ modulo $d=V(m)$. We now use the fact that $mP=m_i p_i$ and ${\displaystyle{\left(\frac{V_{m_i}}{p_i}\right)=1}}$. This puts $m_i$ in certain residue classes modulo 
$$
t(p_i)/\gcd(t(p_i),V(m))=b_{p_i}
$$ 
(analogous to formula (3.23) in \cite{AL}), where this last formula holds because $n\not\in {\mathcal M}_{9}(x)$. 
In our case, we have, by Lemma \ref{lem:5}, 
\begin{eqnarray*}
b_{p_i} & \ge &  \frac{{\text{\rm ord}}_{p_i}(\alpha){\text{\rm ord}}_{p_i}(\beta/\gamma)}{8 a_{p_i}}\ge  \frac{(p_i^{1/2} \exp(c(\log\log z_3)^2))^2}{8z_2^2}\\
& > & p_i \exp(2(\log\log x))^2.
\end{eqnarray*}
The above inequality is the analogue of (3.22) in \cite{AL}. Now the current Lemma \ref{lem:period}, together with the argument from \cite{AL}, shows that the analogue of estimate (3.24) from \cite{AL} 
also holds in our situation. Next the argument from \cite{AL} based on the Chinese Remainder Theorem
leads to the conclusion that
\begin{equation}
\label{eq:N9bound}
\#{\mathcal M}_{10}(x)  \ll \frac{x(\log\log x)^2}{(\log x)^{\lambda\log 2}}.
\end{equation}
On comparing the upper bounds \eqref{eq:N1bound}, \eqref{eq:N2bound}, \eqref{eq:N3bound}, \eqref{eq:N4bound}, \eqref{eq:N5bound}, \eqref{eq:N6bound}, \eqref{eq:N7bound}, \eqref{eq:N8bound}, \eqref{eq:N10bound}
 and \eqref{eq:N9bound} we get that
$$
\#{\mathcal M}(x)\ll \frac{x}{(\log x)^{\min\{\kappa\delta, \nu, \lambda\log 2\}}}.
$$
In order to minimize this upper bound we choose $\kappa$ and $\lambda$ in such a way that $\kappa\delta=\nu=\lambda\log 2$. Thus, $\lambda=\kappa \delta/\log 2$, and we get
$$
\kappa \delta=\left(\frac{1-\kappa}{2}\right)-\frac{\kappa \delta}{\log 2} \log\left(\frac{e(1-\kappa)\log 2}{2\kappa \delta}\right).
$$
Solving we get $\kappa=0.600541$ with $\lambda=0.07452\ldots.$ 
Note that $\lambda<(1-\kappa)/2$ as we required at the outset. The final exponent on the logarithm in 
the saving over the trivial 
bound $\#{\mathcal M}(x)\le x$ is $\kappa \delta=0.0516894\ldots,$ which leads to the desired conclusion.

\section*{Acknowledgements.} 

This work was done while F. L. visited the Max Planck Institute for Mathematics in Bonn in April 2015. This author thanks this institution for its hospitality and support.

\end{document}